\documentclass[a4paper, reqno, 12pt]{amsart}

\usepackage[utf8]{inputenc}
\usepackage{amsthm,amsfonts,amssymb,amsmath,amsxtra,amsrefs}
\usepackage{mathtools}
\usepackage{latexsym}
\usepackage{stmaryrd}
\usepackage{graphicx}
\usepackage{tikz}
\usepackage{tikz-cd}
\usepackage{todonotes,cancel}
\usepackage[margin=1.25in]{geometry}
\usepackage{mathrsfs}
\usepackage{bm}

\usepackage[all]{xy}
\SelectTips{cm}{}
\usepackage{xr-hyper}
\usepackage[colorlinks=false,
   citecolor=Black,
   linkcolor=Red,
   urlcolor=Blue]{hyperref}
\usepackage{verbatim}

\RequirePackage{xspace}
\RequirePackage{etoolbox}
\RequirePackage{varwidth}
\RequirePackage{enumitem}
\RequirePackage{tensor}
\RequirePackage{mathtools}
\RequirePackage{longtable}
\RequirePackage{multirow}

\usepackage{rotating}

\tikzset{
    labl/.style={anchor=south, rotate=90, inner sep=.5mm}
}

\newtheorem{thm}{Theorem}

\newtheorem*{thmintro}{Theorem}
\newtheorem{prop}[thm]{Proposition}

\newtheorem{lem}[thm]{Lemma}

\newtheorem{cor}[thm]{Corollary}

\theoremstyle{definition}

\newtheorem{defi}[thm]{Definition}

\newtheorem{remark}[thm]{Remark}

\numberwithin{equation}{section}
\numberwithin{thm}{section}

\newcommand{\BC}{\ensuremath{\mathbb {C}}\xspace}

\newcommand{{\BG}}{\ensuremath{\mathbb {G}}\xspace}

\newcommand{{\BK}}{\ensuremath{\mathbb {K}}\xspace}

\newcommand{\BN}{\ensuremath{\mathbb {N}}\xspace}

\newcommand{\BP}{\ensuremath{\mathbb {P}}\xspace}
\newcommand{\BQ}{\ensuremath{\mathbb {Q}}\xspace}
\newcommand{\BR}{\ensuremath{\mathbb {R}}\xspace}

\newcommand{\BZ}{\ensuremath{\mathbb {Z}}\xspace}

\newcommand{\CB}{\ensuremath{\mathcal {B}}\xspace}

\newcommand{\CI}{\ensuremath{\mathcal {I}}\xspace}

\newcommand{\CR}{\ensuremath{\mathcal {R}}\xspace}

\newcommand{\CU}{\ensuremath{\mathcal {U}}\xspace}

\newcommand{\fg}{\ensuremath{\mathfrak {g}}\xspace}
\newcommand{\fU}{\ensuremath{\mathfrak {U}}\xspace}

\newcommand{\x}{\times}
\newcommand{\ox}{\otimes}
\newcommand{\bfa}{\textnormal{\textbf{a}}}
\newcommand{\bfi}{\textnormal{\textbf{i}}}

\newcommand{\bfv}{\textnormal{\textbf{v}}}
\newcommand{\bfw}{\textnormal{\textbf{w}}}
\newcommand{\bfB}{\textnormal{\textbf{B}}}
\newcommand{\bfzero}{\textnormal{\textbf{0}}}

\newcommand{\ve}{\varepsilon}
\newcommand{\bfU}{\mathbf{U}}

\begin{document}

\title[]{Partitions of canonical bases}

\author[Jeff York Ye]{Jeff York Ye}
\address{Department of Mathematics, National University of Singapore, Singapore.}
\email{e1124873@u.nus.edu}

\begin{abstract}
We show that various partitions of the canonical basis of quantum groups constructed by Lusztig and by Kashiwara coincide. Using this partition, we show that the subset corresponding to open Richardson varieties equals to the intersection of the subsets corresponding to the Schubert cells. We also show that in type $A$, the weights from open Richardson varieties are saturated in the corresponding Bruhat interval polytope.
\end{abstract}

\maketitle	
\tableofcontents

\section{Introduction}
\label{sec:intro}

\subsection{Canonical bases and total positivity}
\label{sec:intro:CBTP}

Let $G$ be a split connected reductive group over $\BR$ of simply-laced type. Let $\bfU$ be the quantum group associated to the same root data. In \cite{Lus90}, Lusztig constructed a canonical basis $\bfB$ of $\bfU^+$, which has various important properties. For each reduced word $\bfi$ of the longest element $w_0\in W$, there is a parametrization of $\bfB$ by $\BN^{\ell(w_0)}$. For each dominant weight $\lambda$, $\bfB$ restricts to a basis $\bfB(\lambda)$ of the integral highest weight representation $L(\lambda)$ by acting on a highest weight vector $\eta_\lambda$.

Let $\CB=G/B^+$ be the flag variety. It has a decomposition into open Richardson varieties $\CB=\sqcup \mathring{\CB}_{v,w}$ for $v,w\in W$, where each open Richardson variety $\mathring{\CB}_{v,w}$ is defined to be the intersection of the Schubert cell $B^+\dot{w}B^+/B^+$ and the opposite Schubert cell $B^-\dot{v}B^+/B^+$. 

In \cite{Lus94a}, Lusztig defined totally nonnegative monoid $G_{\geq0}$ which has triangular decomposition
\[
G_{\geq0}=U^+_{\geq0}T_{>0}U^-_{\geq0}.
\]
The totally positive part of flag variety $\CB_{\geq0}$ is defined to be the Hausdorff closure of $U^-_{\geq0}B^+/B^+$. The total positive flag variety also admits a compatible decomposition $\CB_{\geq0}=\sqcup \CB_{v,w,>0}$.

In \cite{Lus19}*{Section~10} and \cite{Lus21}*{Section~2.4}, Lusztig introduced two different partitions of $\bfB$ parametrized by the Weyl group $W$ of $G$:
\begin{itemize}
    \item The partition in \cite{Lus19} uses the coordinates of $\bfB$ in $\BN^{\ell(w_0)}$. Lusztig constructed a map $\BN^{\ell(w_0)}$ to the Weyl group $W$, where an element $\bfa$ in $\BN^{\ell(w_0)}$ is send to the Demazure product of simple reflections $s_{i_k}$, over the indices $k$ such that the coordinate $a_k$ in $\BN^{\ell(w_0)}$ is non-zero. The image in $W$ along with the Bruhat order determines the partition $\bfB_w^{\{1\}}$ of $\bfB$.
    \item The partition in \cite{Lus21} uses the canonical basis coefficients of the action of $\CB_{e,w,>0}$ on the highest weight vector $\eta_\lambda\in L(\lambda)$. For any $x\in \CB_{e,w,>0}$, the set $\bfB_w^{tp}(\lambda)$ of canonical basis elements in $L(\lambda)$ having nonzero coefficient in $x\eta_\lambda$ depends only on $W$. If $v\leq w$ in $W$, then $\bfB_v^{tp}(\lambda)\subset \bfB_w^{tp}(\lambda)$. This induces another partition $\bfB_w^{tp}$ of $\bfB$.
\end{itemize}
It is conjectured in \cite{Lus21} that the two partitions are equal.

Meanwhile, we have a third partition by Kashiwara, via Demazure modules:
\begin{itemize}
    \item For each $w\in W$, the extremal weight vector $\eta_{w\lambda}$ generates a $\bfU^+$-submodule $L_w(\lambda)$ of $L(\lambda)$ called the Demazure module. In \cite{Kas93}, Kashiwara showed that the canonical basis for $L(\lambda)$ is also compatible with the Demazure modules. This induces a third partition $\bfB_w^{rep}$ of $\bfB$.
\end{itemize}
This partition was also studied independently by Lusztig in \cite{Lus94b}, by using monomials in the generators to define a subspace $\bfU_w^+$ of $\bfU^+$ and intersecting it with $\bfB$.

\subsection{Main results}
\label{sec:intro:main}

The main goal of this paper is to show that the three partitions of $\bfB$ in \ref{sec:intro:CBTP} coincide, hence proving Lusztig's conjecture in \cite{Lus21}.
\begin{thmintro}
For any $w\in W$, we have
\[
\bfB_w^{\{1\}}=\bfB_w^{tp}=\bfB_w^{rep}.
\]
\end{thmintro}
This is done in Section \ref{sec:bases}, where we identify the three partitions with the same subset of $\BN^{\ell(w_0)}$.

The different characterizations of this partition on $\bfB$ leads to the further results below.

\subsubsection{Richardson compatibility}

In Section~\ref{sec:richardson}, we study the subset $\bfB_{v,w}^{tp}(\lambda)$ defined using the action of $\CB_{v,w,>0}$ on $L(\lambda)$. Assuming $\lambda$ is regular, we show that it is equal to the intersection of the canonical bases for the corresponding Demazure and opposite Demazure modules. This is parallel to the definition of open Richardson varieties as intersections of Schubert and opposite Schubert cells, and give further evidence to the deep relation between total positivity and canonical bases discovered by Lusztig.

\subsubsection{Saturation of Bruhat interval polytopes}

For each dominant weight $\lambda$, the weights of the weight vectors in $L(\lambda)$ define a subset of the weight lattice $X$. Its convex hull in $X_\BR$ is a polytope $P(\lambda)$ with vertices given by the extremal weights $w\lambda$ for $w\in W$. It is well known that the weights of $L(\lambda)$ is saturated, i.e. any $\mu\in X$ is a weight of $L(\lambda)$ if and only if $\mu\in P(\lambda)$ and $\lambda-\mu$ is in the root lattice.

The Bruhat interval polytope $P_{v,w}(\lambda)$ is defined as the convex hull of $u\lambda$ for $v\leq u\leq w$ in $X_\BR$ \cite{KW15}. An alternative definition of the polytope involves the moment map of the torus action on $\CB$, see \citelist{\cite{Ati82}\cite{GS82}}.

Bruhat interval polytopes of the form $P_{e,w}(\lambda)$ were studied by Besson, Jeralds, and Kiers in \cite{BJK}. They proved that the saturation extends to the Demazure setting for all finite types except type $E$.

In Section \ref{sec:saturation}, we extend the saturation of $P(\lambda)$ and $P_{e,w}(\lambda)$ to $P_{v,w}(\lambda)$ for general $v\leq w\in W$ in type $A$ and regular $\lambda$, where the weights of $L(\lambda)$ are replaced by the weights of $\bfB_{v,w}^{tp}(\lambda)$. This gives an interesting application for the Richardson compatibility.

\vspace{.2cm}

\noindent {\bf Acknowledgment: } The author would like to thank Huanchen Bao for valuable discussions and suggestions.

\section{Preliminaries}\label{sec:prelim}        

\subsection{The monoid $G_{\ge 0}$}
\label{sec:prelim:basic}

Let $G$ be a split connected reductive group over $\BR$. We shall identify $G$ with the set of $\BR$-rational points $G(\BR)$. We fix a pinning $(T,B^+,B^-,x_i,y_i;i\in I)$ of $G$ following \cite{Lus94a}. Let $U^\pm$ be the unipotent radicals of $B^{\pm}$, respectively. Let $X$ be the weight lattice and $Y$ be the coweight lattice. Let $\{\alpha_i\}_{i\in I}$ be the simple positive roots, $\{\alpha_i^\vee\}_{i\in I}$ the corresponding simple positive coroots, $\{\omega_i\}_{i\in I}$ the fundamental weights, and $\{\omega_i^\vee\}_{i\in I}$ the fundamental coweights. Let $A=(a_{ij})_{i,j\in I}$ be the Cartan matrix. We shall assume $G$ is of simply-laced type throughout this paper.

Let $W$ be the Weyl group of $G$, with simple reflections $s_i\in W$ for $i\in I$. Let $w_0$ be the longest element of $W$ with length $m=\ell(w_0)$. Let $\CI$ be the set of reduced words of $w_0$, i.e. sequences $\bfi=(i_1,\cdots,i_m)$ of elements in $I$ such that $s_{i_1}\cdots s_{i_m}=w_0$. Let $\le$ be the Bruhat order on $W$.

For $i\in I$, let $\dot{s}_i=x_i(1)y_i(-1)x_i(1)\in G$. For any $w\in W$ with reduced expression $s_{i_1}\cdots s_{i_k}$, let $\dot{w}=\dot{s}_{i_1}\cdots \dot{s}_{i_k}$.  It is known that $\dot{w}$ is independent of the choice of the reduced expressions.
        
Let $U^-_{\geq0}$ be the monoid in $U^-$ generated by $y_i(a)$ for $i\in I$, $a\in \BR_{>0}$.  Let $U^+_{\geq0}$ be the monoid in $U^+$ generated by $x_i(a)$ for $i\in I$, $a\in \BR_{>0}$. Let $T_{>0}$ be the identity component of $T(\BR)$. Let $G_{\ge 0}$ be the submonoid of $G$ generated by $U^{\pm}_{\ge 0}$ and $T_{>0}$. It follows from \cite{Lus94a} that $G_{\ge 0} = U^{+}_{\ge 0} T_{>0}U^{-}_{\ge 0} = U^{-}_{\ge 0} T_{>0}U^{+}_{\ge 0}$.

\subsection{Canonical bases}
\label{sec:prelim:bases}

Let $\bfU$ be the quantum group associated to the root data of the reductive group $G$. It is a $\BQ(q)$-algebra generated by $E_i$, $F_i$ for $i\in I$ and $K_\mu$ for $\mu \in Y$. We denote  by $E_{i}^{(n)}$ and $F_i^{(n)}$ the divided powers defined in \cite{Lus93}*{\S3.1}.  Let $\mathbf{U}^+$ (resp. $\mathbf{U}^-$) be the subalgebra of $\mathbf{U}$ generated by $E_i$ ($i\in I$) (resp. $F_i$ ($i\in I$)). Let $\bfB$ be the canonical basis of $\bfU^-$.  

Let $X^+$ be the set of dominant weights. For any $\lambda \in X^+$, let $L(\lambda)$ be the integrable highest weight $\bfU$-module  defined in \cite{Lus93}*{Proposition~3.5.6}. Let $\eta_\lambda$ be the highest weight vector. Let $\bfB(\lambda)$ be the subset of $b\in \bfB$ such that $b\eta_\lambda\neq 0$. By \cite{Lus90}*{Section~8}, the set of elements $b\eta_\lambda$ over $b\in \bfB(\lambda)$ form a basis of $L(\lambda)$, which is called the canonical basis of $L(\lambda)$. Via the identification $b\mapsto b\eta_\lambda$, we will abuse notation and view $\bfB(\lambda)$ also as a subset of $L(\lambda)$.

Let $w\in W$. We define $L_w(\lambda) = \bfU^+ \eta_{w\lambda} \subset L(\lambda)$ and $\bfB_w^{rep}(\lambda) = \bfB(\lambda) \cap L_w(\lambda)$. It follows from \cite{Kas94}*{Lemma~8.2.1} that $\bfB_w^{rep}(\lambda)$ is a basis of $L_w(\lambda)$.

Kashiwara defined crystal structures on $\bfB(\lambda)$ and $\bfB$ in \cite{Kas94}. The crystal structure comes with Kashiwara operators $\Tilde{e}_i$ and $\Tilde{f}_i$ for $i\in I$ sending $\bfB(\lambda)$ (resp. $\bfB$) to $\bfB(\lambda)\cup \{0\}$ (resp. $\bfB \cup \{0\}$). 

\begin{remark}
To be precise, the crystal structure is defined on the limit at $q=\infty$ of the basis  $\bfB$. Since this limit is in natural bijection with $\bfB$, we shall assume the crystal structure is given on $\bfB$.
\end{remark}

\subsection{Tropical parametrization of the canonical basis}
\label{sec:prelim:tropflag}
        
Following \cite{Lus97}, we consider a parametrization of $\bfB$ using tropical geometry. 
        
For any $(\bfi,\bfa)$ and $(\bfi',\bfa')$ in $\CI\x \BZ^m$, we say that they are \textit{adjacent} if one of the two following conditions hold:

\begin{itemize}
    \item $\bfi'$ is obtained from $\bfi$ by replacing two consecutive indices $(i,j)$ (with $s_i s_j =s_j s_i$) by $(j,i)$ such that $\bfa'$ is obtained from $\bfa$ by replacing the two consecutive coordinates $(a,b)$ if $\bfa$ (corresponding to the positions in which $i,j$ appear in $\bfi$) by $(b,a)$;
    
    \item  $\bfi'$ is obtained from $\bfi$ by replacing three consecutive indices $(i,j,i)$ (with $s_i s_j s_i =s_j s_i s_j$) by $(j,i,j)$ such that $\bfa'$ is obtained from $\bfa$ by replacing the three consecutive coordinates $(a,b,c)$ if $\bfa$ (corresponding to the positions in which $i,j,i$ appear in $\bfi$) by $(a',b',c')$, where $a' = b + c - \min(a,c)$, $b' = \min (a,c)$, $c'= a+b - \min(a,c)$.
\end{itemize}

Let $\CU_\BZ$ be the set of equivalence classes on $\CI\x \BZ^m$ for the equivalence relation defined using the adjacency relations above. Since any two reduced words of $w_0$ are related by braid moves $(ij)\sim(ji)$ and $(iji)\sim(jij)$, each reduced word $\bfi\in \CI$ gives a bijection $\CU_\BZ\xrightarrow{\sim} \BZ^m$. Furthermore, the subset $\CI\x \BN^m$ in $\CI\x \BZ^m$ is closed under adjacency relations, so we get a subset $\CU_\BN\subset \CU_\BZ$. This set $\CU_\BN$ is first considered in \cite{Lus90}*{Section~2.1}.

For any $i\in I$, fix any reduced word $\bfi = (i_1, \dots, i_m)\in \CI$ with $i_1=i$. We have the following well-defined maps by \cite{Lus97}*{\S2.2 \& \S 2.3}:
\begin{itemize}
    \item the map $\ve_i:\CU_\BZ\to \BZ$ given by $\ve_i(\bfi,\bfa)=a_1$;
    \item for any $c\in \BZ$, the map $T_{i,c} : \CU_{\BZ} \rightarrow \CU_{\BZ}$ given by
    \[
    T_{i,c}(\bfi,\bfa)=(\bfi,(a_1+c,a_2,\cdots,a_m)).
    \]            
\end{itemize}

We have an explicit bijection $u : \CU_{\BN} \xrightarrow{\sim} \bfB$ following \cite{Lus93}*{\S42} such that 
\begin{itemize}
    \item $u((\bfi, (0, \dots, 0))) = 1$ is the identity element in $\bfB$;
    \item   $\tilde{f}_i u (\zeta) = u (T_{i,1} \zeta)$;
    \item   $\tilde{e}_i u (\zeta) = u (T_{i,-1} \zeta)$ if $\varepsilon_i(\zeta) \ge 1$;
    \item   $\tilde{e}_i u (\zeta) = 0$ if $\varepsilon_i(\zeta) =0$.
\end{itemize}
For any $\lambda \in X^+$, let $\CU_{\BN, \lambda}$ be the subset of $\CU_{\BN}$ in bijection with $\bfB(\lambda)$ via the bijection $u$.

\subsection{Flag varieties}
\label{sec:prelim:flag}

Let $\CB =G/B^+$ be the flag variety of $G$. For any $ w \in W$, we write $\mathring{\CB}_w = B^+ \dot{w} B^+ / B^+$ and $\mathring{\CB}^w = B^- \dot{w} B^+ /B^+$. We denote by ${\CB}_w$ and $\CB^w$ the Zariski closure of $\mathring{\CB}_w$ and $\mathring{\CB}^w $ in $\CB$, respectively. For any $ v, w \in W$, we set $\mathring{\CB}_{v,w} = \mathring{\CB}_w \cap \mathring{\CB}^v$. This is the open Richardson variety. It is known that the intersection if non-empty is and only if $v \le w$.  We have the decomposition
\[
\CB=\bigsqcup_{v\leq w} \mathring{\CB}_{v,w}.
\]
        
Let $\CB_{\geq0}$ be the closure of $U^-_{\geq0}B^+/B^+$ in $\CB$ with respect to the Hausdorff topology. Let $\CB_{v,w,>0}=\mathring{\CB}_{v,w}\cap \CB_{\geq0}$.

\subsection{Deodhar components}
\label{sec:prelim:Deodhar}

We follow \cite{MR04}. For $w\in W$, an \textit{expression} for $w$ is a finite sequence $\bfw=(i_1,\cdots,i_n)$ in $I$ such that $w=s_{i_1}\cdots s_{i_n}$. By abuse of notations, we also write the expression as $\bfw=s_{i_1}\cdots s_{i_n}$. We define the length $\ell(\bfw)=n$. The expression $\bfw$ is said to be \textit{reduced} if the length $n$ is minimal among all expressions for $\bfw$, in which case we we have $\ell(w)=\ell(\bfw)$.

A \textit{subexpression} of $\bfw$ is a sequence $\bfv=(i'_1,\cdots,i'_n)$ such that $i'_k\in \{i_k,e\}$ for all $k$. 
We further write $v=s'_{i_1}\cdots s'_{i_n}\in W$, where $s'_{i_k}=s_{i_k}$ if $i'_k=i_k$, and $s'_{i_k}=e$ if $i'_k=e$. We say that $\bfv$ is a subexpression for $v$ in $\bfw$.

Let $\bfw=(i_1,\cdots,i_n)$ be a reduced expression for $w$. For $v\leq w$, we define the \textit{right-most} subexpression $\bfv$ for $v$ in $\bfw$ as follows:
\begin{itemize}
\item If $vs_{i_n}<v$, then $i'_n=i_n$ and we let $(i'_1,\cdots,i'_{n-1})$ to be the right-most subexpression for $vs_{i_n}$ in $(i_1,\cdots,i_{n-1})$.
\item If $vs_{i_n}>v$, then $i'_n=e$ and we let $(i'_1,\cdots,i'_{n-1})$ to be the right-most subexpression for $v$ in $(i_1,\cdots,i_{n-1})$.
\end{itemize}

For any right-most subexpression $\bfv$ in $\bfw$, let $J^+$ be set of the indices $k$ such that $i'_k=i_k$, and $J^\circ$ be the set of indices $k$ such that $i_k=e$. Consider the subset $\CR_{\bfv,\bfw}$ in $\CB$ given by $\CR_{\bfv,\bfw}=\{g_1\cdots g_nB^+\}$, where $g_k=\dot{s}_{i_k}$ if $k\in J^+$, and $g_k=y_{i_k}(t_k)$ if $k\in J^\circ$, for some $t_k\in \BC^\x$. This defines a map $(\BC^\x)^{|J^\circ|}\to \CR_{\bfv,\bfw}$, which is an isomorphism by \cite{Deo85}*{Theorem 1.1}.

Marsh and Rietsch showed that this induces a parametrization of $\CB_{v,w,>0}$. 

\begin{thm}{\cite{MR04}*{Theorem 11.3}}
\label{MRparametrization}
Let $\bfw$ be a reduced expression for $w$ and $\bfv$ the right-most subexpression for $v$ in $\bfw$. The map $(\BC^\x)^{|J^\circ|}\to \CR_{\bfv,\bfw}$ restricts to an isomorphism $\BR_{>0}^{|J^\circ|}\to \CB_{v,w,>0}$.
\end{thm}

\section{Stratifications of the canonical bases}
\label{sec:bases}

In this section, we recall several stratifications of the canonical basis $\bfB$ of $U_q^-(\fg)$, or equivalently, the set $\CU_{\BN}$ by Section \ref{sec:prelim:tropflag}. We show along the way that they are in fact equivalent.

\subsection{Stratification by representation theory}
\label{sec:bases:repn}
        
For any $w\in W$ and $\lambda \in X^+$, let $\eta_{w\lambda}\in L(\lambda)$ be the unique canonical basis element of weight $w\lambda$. We define the Demazure module $L_w(\lambda) =  U_q^+(\fg) \eta_{w\lambda} \subset L(\lambda)$ as a $U_q^+(\fg)$-module. We have $L_v(\lambda) \subset L_w(\lambda)$ if and only if $v \le w$.

Let $e \in \bfB$ be the identity element of $U_q^-(\fg)$. For any $w \in W$, Kashiwara \cite{Kas93}*{Proposition~3.2.5} defined a subset $\bfB_w^{rep} \subset \bfB$ by 
\[
\bfB_w^{rep} = \{\tilde{f}_{i_1}^{a_{1}}\tilde{f}_{i_2}^{a_{2}}\cdots \tilde{f}_{i_n}^{a_{n}} (e) \,\vert\, s_{i_1} s_{i_2} \cdots s_{i_n} \text{ is a reduced expression of }w, a_i \in \BZ_{\ge 0}\}.
\]
It was further shown that the action map $U_q^-(\fg) \rightarrow L(\lambda)$, $u \mapsto u \eta_\lambda$, maps $\bfB_w^{rep}$ to the canonical basis $\bfB_w^{rep}(\lambda) \subset \bfB(\lambda)$ of $L_w(\lambda)$ or $0$.

We note that $\bfB_w^{rep}$ is the same as the partition defined by Lusztig in \cite{Lus94b}*{Section~5.3}. See \cite{Lus94b}*{Proposition~4.2} and \cite{Kas93}*{Corollary~3.2.2} for more details on the equivalence.

We define $\CU_w^{rep}$ be the image of $ \bfB_w^{rep}$ under the bijection $\bfB \leftrightarrow \CU_{\BN}$. Therefore we conclude from \cite{Kas93}*{Proposition~3.2.5} that

\begin{enumerate}[label=(\alph*)]
    \item $\CU_{e}^{rep} = \{\bfzero\}$, where $\bfzero=(\bfi,(0,\cdots,0))$;
    \item if $w < s_i w$, then we have $     \CU_{s_iw}^{rep}=\bigcup_{k\geq0} \Tilde{f}_i^k \CU_{w,rep}$;
    \item if $v \le w$, then we have $\CU_v^{rep} \subset \CU_w^{rep}$.
\end{enumerate}
        
Recall the bijection $f_{\bfi}: \CU_{\BN} \xrightarrow{\sim} \BN^{\ell(w_0)}$ for any reduced expression $\bfi \in \CI$ of the longest element $w_0$.  Let $\CU_{w, \bfi}^{rep}$ be the image of $\CU_w^{rep}$ under the bijection. For any subset $A  \subset \{1, \dots, m = \ell(w_0)\}$, we define the coordinate cone $\CU_{A, \bfi}^{rep} \subset \BN^{m}$ by 
\[
\CU_{A, \bfi}^{rep} = \{(x_1,\cdots, x_m) \in \BN^m\,|\, x_j=0 \text{ for } j\in A\}.
\]
Let $w \in W$ and $\bfi \in \CI$. Let $S_{w,\bfi}$ be the set of reduced subexpression of $ww_0$ in $\bfi$. We shall identify elements in $S_{w,\bfi}$ with subsets of $\{1, \dots ,m\}$, by taking the indices $k$ such that $i_k'\neq e$.

\begin{lem}\label{lem:Ssw}
Let $w \in W$ be such that $s_iw \ge w$ for some $i \in I$. Let $\bfi = (i_1, \dots, i_m)$ be a reduced expression of $w_0$ such that $i_1 =i$. We have
\[
S_{s_iw, \bfi} = \{A - \{1\} \,\vert\, A \in S_{w, \bfi}, 1  \in A\} 
\]
\end{lem}

\begin{proof}
By definition, $A \subset \bfi$ is a reduced subexpression of $ww_0$ in $\bfi$. If $i_1 \in A$, then $A - \{1\}$ is a reduced subexpression of $s_{i}ww_0 = s_{i_1}ww_0$. This shows      $S_{s_iw, \bfi} \supset \{A - \{1\} \,\vert\, A \in S_{w, \bfi}, 1 \in A\}$.

Let  $A' \in S_{s_iw, \bfi}$. Then $A'$ be a reduced subexpression of $s_{i}ww_0 = s_{i_1}ww_0$ in $\bfi$. Since we have $s_i (s_iww_0)= ww_0 > s_iww_0$, we must have $i_1 \not \in A'$ by our choice of $\bfi$. Let $A = A' \cup \{1\}$. We see that $A \in S_{w, \bfi}$. This shows  $S_{s_iw, \bfi} \subset \{A - \{1\} \,\vert\, A \in S_{w, \bfi}, 1 \in A\}$.

This finishes the proof.
\end{proof}

\begin{prop}
\label{demcone}
We have
\[
\CU_{w, \bfi}^{rep} = \bigcup_{A \in S_{w, \bfi}}\CU_{A, \bfi}^{rep}
\]
\end{prop}

\begin{proof}
We proceed by induction on $\ell(w)$. The base case when $w =e$ is trivial. The case when $w = w_0$ is also trivial. We can assume $w \neq e $ and $w \neq w_0$. 
    
Assume the statement holds for all $w' \le w$ with respect to all reduced expressions of $w_0$. Let $s_iw > w$ for some $i \in I$. 
    
(a) Let $\bfi = (i_1, \dots, i_m)$ be a reduced expression $w_0$ such that $i_1 =i$. We claim $\CU_{s_iw, \bfi}^{rep}  = \bigcup_{A \in S_{s_iw, \bfi}} \CU_{A, \bfi}^{rep}$.
    
By definition and the induction hypothesis, we have 
\[
\CU_{s_iw, \bfi}^{rep}= \bigcup_{k \ge0} \tilde{f}_i^k\CU_{w, \bfi}^{rep} = \bigcup_{A \in S_{w, \bfi}}\big(\bigcup_{k \ge0} \tilde{f}_i^k\CU_{A, \bfi}^{rep}\big).
\]
Note that $\tilde{f}_i^k (\bfi, (a_1, a_2, \dots, a_m)) = (\bfi, (a_1+k, a_2, \dots, a_m))$ due to our choice of $\bfi$. We conclude that  $\cup_{k \ge0} \tilde{f}_i^k\CU_{A, \bfi, rep} = \CU_{A - \{1\}, \bfi, rep}$. Hence by Lemma~\ref{lem:Ssw},  we have   
\[
\CU_{s_iw, \bfi}^{rep} = \bigcup_{A \in S_{w, \bfi}} \CU_{A -\{1\}, \bfi}^{rep} \supset \bigcup_{A \in S_{w, \bfi}, 1 \in A} \CU_{A -\{1\}, \bfi}^{rep} = \bigcup_{A \in S_{s_iw, \bfi}} \CU_{A, \bfi}^{rep}.
\]

Let $A \in S_{w, \bfi}$ be such that $1 \not\in A$, we claim that $\CU_{A, \bfi}^{rep} \subset \CU_{A' -\{1\}, \bfi}^{rep}$ for some $A'\in S_{w, \bfi}$ with $1   \in A'$. Note that $\CU_{A, \bfi}^{rep} \subset \CU_{A' -\{1\}, \bfi}^{rep}$ is equivalent to  $A' - \{1\} \subset A$. Assume $A = \{i_{j_1},i_{j_2}, \cdots, i_{j_l}\}$, such that $s_{i_{j_1}}s_{i_{j_2}} \cdots s_{i_{j_l}}$ is a reduced subexpression of $ww_0$ in $\bfi$. Since $s_{i_1} ww_0 < ww_0$, we have $s_{i_1}\cdot s_{i_{j_1}}s_{i_{j_2}} \cdots s_{i_{j_l}} = s_{i_{j_1}}s_{i_{j_2}} \cdots \widehat{s_{i_{j_a}}} \cdots s_{i_{j_l}}$ by the exchange property of Coxeter groups. Then $A' = \{i_1, i_{j_1}, \dots, \widehat{i_{j_a}}, \dots, i_{j_l}\}$ satisfies the desired properties. We now finishes the proof of the claim (a). 

(b) We claim $\CU_{s_iw, \bfi}^{rep}  = \bigcup_{A \in S_{s_iw, \bfi}} \CU_{A, \bfi}^{rep}$ for any $\bfi \in \CI$.

Let $\bfi$ and $\bfi'$ be two reduced expressions in $\CI$. Recall the transition map
\[
\BN^{\ell(w_0)} \xrightarrow{f_\bfi^{-1}} \CU_{\BN} \xrightarrow{f_{\bfi'}} \BN^{\ell(w_0)}.
\]
It suffices to show $f_{\bfi'} \circ f_\bfi^{-1}$ restricts to a bijection 
\[
\bigcup_{A \in S_{s_iw, \bfi}} \CU_{A, \bfi}^{rep} \leftrightarrow \bigcup_{A \in S_{s_iw, \bfi'}} \CU_{A, \bfi'}^{rep}.
\]    
    
Since any reduced expressions of $w_0$ are related by braid relations,it suffices to assume $\bfi$ and $\bfi'$ are related by a single braid move.   

If $\bfi = (\cdots, i,j, \cdots)$, $\bfi' = (\cdots, j,i, \cdots)$, then the map $f_{\bfi'} \circ f_\bfi^{-1}$ simply permutes coordinates. The bijection is trivial.
 
Assume $\bfi = (\cdots, i,j, i, \cdots)$, $\bfi' = (\cdots, j,i,j, \cdots)$. Then the map $f_{\bfi'} \circ f_\bfi^{-1}$ gives the following bijection (locally)

\begin{center}
    \begin{tabular}{ |c|c| } 
    \hline
    $\bfi=(\dots, i,j,i, \dots)$ & $\bfi' = (\dots, j,i,j, \dots)$ \\ 
    \hline
    $(\cdots, 0,0,0, \cdots )$ & $(\cdots, 0,0,0, \cdots)$ \\
    \hline
    $(\cdots, \BN,0,0, \cdots)$ & $(\cdots, 0,0,\BN, \cdots)$ \\
    \hline
    $(\cdots, \BN,\BN,0, \cdots) \cup (\cdots, 0,\BN,\BN,\cdots)$ & $(\cdots, \BN, 0, \BN,\cdots)$ \\
    \hline
    $(\cdots,\BN,\BN,\BN,\cdots)$ & $(\cdots, \BN,\BN,\BN, \cdots)$ \\
    \hline
    \end{tabular}
\end{center}
The claim (b) follows now.
\end{proof}

\subsection{Stratification by the semifield of $\{1\}$}
\label{sec:bases:U1strat}

We follow \cite{Lus19} to introduce another stratification of $\CU_{\BN}$. For any semifield $K$, let $\fU(K)$ be the monoid generated by $x_i(a)$ for $i\in I$, $a\in K$ subject to the following relations:

\begin{itemize}
    \item $x_i(a)x_i(b)=x_i(a+b)$;
    \item for $a_{ij}=0$, we have $x_i(a)x_j(b)=x_j(b)x_i(a)$;
    \item for $a_{ij}=-1$, we have $x_i(a)x_j(b)x_i(c)=x_j(bc/(a+c))x_i(a+c)x_j(ab/(a+c))$.
\end{itemize}
        
Let $\{1\}$ be the (unique) semifield of one element. Let $\fU(\{1\})$ be the monoid generated by $s_i=x_i(1)$ ($i\in I$) with multiplication denoted by $\ast$. Then $\fU(\{1\})$ is isomorphic to the monoid $(W, \ast)$, where $\ast$ denotes the Demazure product on $W$.

By \cite{Lus19}, there is a well-defined map $\CU_\BN\to \fU(\{1\})$ mapping $(\bfi,\bfa)\mapsto s_{j_1} \ast \cdots \ast s_{j_k}$, where $j_1<\cdots<j_k$ is the sequence of all $j\in \{1,\cdots,r\}$ such that $a_j=0$. 

We define $\chi:\CU_\BN\to \fU(\{1\}) \xrightarrow{\sim} W$ and 
\[
\CU_w^{\{1\}} = \{(\bfi,\bfa) \in \CU_\BN \,\vert\, \chi((\bfi,\bfa)) \ge ww_0 \text{ in the Bruhat order of $W$}. \}
\]

\begin{remark}
\label{U1twistremark}
Note that the stratification is twisted by $w_0$ from the original definition in \cite{Lus19}.
\end{remark}

By Proposition \ref{demcone} and the definitions of $S_{w,\bfi}$ and $\CU_{w, \{1\}}$, we obtain the following corollary.

\begin{cor}
\label{cor:rep_1_equal}
We have $\CU_w^{rep} = \CU_w^{\{1\}}$.
\end{cor}

\subsection{Stratification by total positivity}
\label{sec:bases:TP}

We follow \cite{Lus21} to introduce a stratification of $\CU_{\BN}$ using totally nonnegative flag varieties. Consider the specialization of the integral form of $L(\lambda)$ at $q=1$ and extend the scalar to $\BR$,  we obtain a simple $G$-module and $\fg$-module. We will abuse notation and still denote it by $L(\lambda)$. We also get the canonical basis $\bfB(\lambda)$ of $L(\lambda)$ after specialization. The set $\bfB(\lambda)$ is still identified with a subset of $\bfB$ and equivalently $\CU_\BN$.

Let $\lambda \in X^+$ and $\eta_\lambda$ be the highest weight vector in $L(\lambda)$. The line $[\eta_\lambda]\in \BP(L(\lambda))$ is fixed by $B^+$, so we get a map $\Psi:\CB\to \BP(L(\lambda))$ given by $xB^+\mapsto [x\eta_\lambda]$.
        
Let $v \le w \in W$ and $x \in \CB_{v,w,>0}$. We write $\Psi(x) = [\sum_{b \in \bfB(\lambda)}c_b(x) b]$. We define 
\[
\bfB_{v,w}^{tp}(\lambda) = \{ b \in \bfB(\lambda) \,\vert\, c_b(x) \neq 0\}.
\]

\begin{lem}
\label{lem:tp_indep}
The definition $\bfB_{v,w}^{tp}(\lambda)$ is independent of the choice of $x \in \CB_{v,w,>0}$.     
\end{lem}

\begin{proof}
By the Marsh-Rietsch parametrization of $\CB_{v,w,>0}$, $x=g_1\cdots g_nB^+$, where $g_k=\dot{s}_{i_k}$ if $k\in J^+$ and $g=y_{i_k}(t_k)$ if $k\in J^\circ$. Then up to multiplication by $\BR^\x$, the coefficient $c_b(x)$ is a subtraction-free expressions of $t_k$. Therefore $c_b(x) \neq 0$ for some $x \in \CB_{v,w,>0}$ if and only if $c_b(x) \neq 0$ for all $x \in \CB_{v,w,>0}$.
\end{proof}

Let $\CU_{v,w}^{\lambda, tp}$ be the set in $\CU_\BN$ corresponding to $\bfB_{v,w}^{tp}(\lambda)$. We write $\CU_{w, \lambda}^{tp} =\CU_{e,w, \lambda}^{tp}$ and $\CU_w^{tp} = \cup_{\lambda \in X^+}\CU_{w, \lambda}^{tp}$. This union is best viewed as a limit as $\lambda\to \infty$: if $\lambda-\mu\in X^+$, then $\CU_{w, \mu}^{tp} \subset \CU_{w, \lambda}^{tp}$. This is parallel to the fact that $\bfB$ is a union of $\bfB(\lambda)$ for $\lambda\in X^+$.

\begin{prop}
\label{tp=rep}
We have $\CU_w^{tp} = \CU_w^{rep}$.
\end{prop}
    
\begin{proof}
It suffices to prove $\CU_{w, \lambda}^{tp} = \CU_{w, \lambda}^{rep}$ for any $\lambda \in X^+$. 
    
Let $\lambda \in X^+$ and $w = s_{i_1} \cdots s_{i_n}$ be a reduced expression of $w$. Then by the Marsh--Rietsch parametrization, we have $y_{i_1}(a_1) \cdots y_{i_n}(a_n) B^+ \in \CB_{e,w, >0}$ if $a_i >0$ for all $i$. Then by \cite{Kas93}*{Section 3}, we have
\begin{align*}
&y_{i_1}(a_1) \cdots y_{i_n}(a_n) \eta_{\lambda} = (\sum_{j=0}^{\infty} a_1^j F_{i_1}^{(j)}) \cdots (\sum_{j=0}^{\infty} a_n^j F_{i_n}^{(j)}) \eta_\lambda \\
&= \sum_{} a_1^{j_1} \cdots a_n^{j_n} c_{j_1, \dots, j_n}G(\tilde{f}_{i_1}^{j_1} \cdots \tilde{f}_{i_n}^{j_n} \eta_\lambda) + \sum_{b \in B(\lambda)} c_b b, \quad \text{for } c_{j_1, \dots, j_n} >0, c_b \ge 0.
\end{align*}
It follows that $\CU_{w, \lambda}^{tp} \supset \CU_{w, \lambda}^{rep}$. The reverse inclusion follows from \cite{Kas93}*{Corollary~3.2.2}.
\end{proof}

Combining with Corollary~\ref{cor:rep_1_equal}, this proves Lusztig's conjecture in \cite{Lus21}*{Section 2.4}. Conceptually, we use Kashiwara's Demazure module partition as a bridge to prove the equality of the two Lusztig's partitions.

For the remaining of the paper, we will drop the superscripts $rep$, $tp$ and $\{1\}$ for $\bfB_w$ and $\CU_w$.

\section{Compatibility of Richardson varieties}
\label{sec:richardson}

In this section, we show that the $\bfB_{v,w}^{tp}(\lambda)$ defined in Section \ref{sec:bases:TP} is an intersection of the canonical bases of the corresponding Demazure and opposite Demazure modules. This parallels the geometric definition of Richardson varieties as intersections of Schubert varieties and opposite Schubert varieties.

\subsection{Maps between totally positive parts}
\label{sec:richardson:Deodhar}

We will use following proposition repeatedly. Its proof is immediate from Theorem \ref{MRparametrization}.

\begin{prop}
\label{MultiplyBijection}
If $w<s_iw$, then $Y_i:\BR_{>0}\x \CB_{v,w,>0}\to \CB_{v,s_iw,>0}$ given by $(a,gB^+)\mapsto (y_i(a)gB^+)$ is a bijection.

Similarly, if $s_iv<v$, then $X_i:\BR_{>0}\x \CB_{v,w,>0}\to \CB_{s_iv,w,>0}$ given by $(a,gB^+)\mapsto (x_i(a)gB^+)$ is a bijection.
\end{prop}

\begin{proof}
Fix a reduced expression $\bfw$ for $w$, giving the positive subexpression $\bfv$ for $v$ in $\bfw$. Then if $s_iw>w$, $\bfw'=(e,s_i,s_iw_{(1)},\cdots,s_iw_{(n)})$ is a reduced expression for $s_iw$, and $(e,e,v_{(1)},\cdots,v_{(n)})$ is the positive subexpression for $v$ in $\bfw'$. The desired bijection follows immediately after writing down the parametrizations explicitly.

For $X_i$, we use the automorphism $\phi$ of $G$ such that $\phi(x_i(a))=y_i(a)$, $\phi(y_i(a))=x_i(a)$ and $\phi(t)=t^{-1}$ for $t\in T$. Then $\phi$ preserves $\CB_{\geq0}$ and gives an isomorphism between $\CB_{v,w}$ and $\CB_{ww_0,vw_0}$. Thus $\phi$ gives a bijection between $\CB_{v,w,>0}$ and $\CB_{ww_0,vw_0,>0}$, so $X_i$ is a bijection by combining $\phi$ with $Y_i$.
\end{proof}

\subsection{Richardson stratification}
\label{sec:richardson:strat}

Recall the subset $\bfB_{v,w}^{tp}(\lambda)$ defined in Section \ref{sec:bases:TP} using total positivity. By the results in Section \ref{sec:bases}, $\bfB_{e,w}^{tp}(\lambda)=\bfB_w(\lambda)$.

The following is the main theorem of this section.

\begin{thm}
\label{thm:intersection}
Assume $\lambda$ is regular. For any $v,w\in W$ with $v\leq w$, we have
\[
\bfB_{v,w}^{tp}(\lambda)=\bfB_{e,w}^{tp}(\lambda)\cap \bfB_{v,w_0}^{tp}(\lambda).
\]
\end{thm}

For any $w\in W$, the opposite Demazure module is defined to be the $U_q^-(\fg)$-submodule of $L(\lambda)$ generated by $\eta_{w\lambda}$. Let $\bfB^w(\lambda)$ be its canonical basis.

\begin{prop}
\label{OppositeDemazure}
For any $w\in W$, we have $\bfB^w(\lambda)=\bfB_{w,w_0}^{tp}(\lambda)$.
\end{prop}

\begin{proof}
Let $w=s_{i_1}\cdots s_{i_k}w_0$. By Proposition \ref{MultiplyBijection}, any $gB^+\in \CB_{w,w_0,>0}$ can be written as
\[
gB^+=x_{i_1}(a_1)\cdots x_{i_k}(a_k) \dot{w}_0B^+
\]
for $a_i\in \BR_{>0}$. The proof is then the same as that in Proposition \ref{tp=rep}.
\end{proof}

We then prove a few lemmas on the crystal structure of $\bfB(\lambda)$. The next two Propositions by Kashiwara are crucial.

\begin{prop}[\cite{Kas94}*{Proposition~3.3.5}]
\label{DemazureString}
For any $i$-string $S$ of $\bfB(\lambda)$, $\bfB_w(\lambda)\cap S$ is either empty, $S$, or the highest weight vector $\{b\}$. If $s_iw<w$, then $\bfB_w(\lambda)\cap S$ is either empty or $S$.

Similarly, $\bfB^v(\lambda)\cap S$ is either empty, $S$, or the lowest weight vector $\{b'\}$. If $s_iv>v$, then $\bfB^v(\lambda)\cap S$ is either empty or $S$.
\end{prop}

\begin{prop}[\cite{Kas94}*{Proposition~4.4}]
\label{CrystalIntersection}
$\bfB_w(\lambda)\cap \bfB^v(\lambda)$ is non-empty if and only if $v\leq w$.
\end{prop}
        
\begin{lem}
\label{LowestWtCoeffZero}
Let $x\in L(\lambda)$. For any $i\in I$, let $E_ix=\sum_{b\in \bfB(\lambda)} c_bb$. Then if $b$ is a lowest weight vector in an $i$-string, i.e. $\Tilde{f}_ib=0$, we must have $c_b=0$.
\end{lem}

\begin{proof}
We may assume that $x$ is a weight vector. Recall that any weight vector $u\in L(\lambda)$ of weight $\mu$ can be uniquely written as
\[
u=u_0+E_iu_1+\cdots E_i^{(N)}u_N,
\]
where each $u_k$ is a weight vector of weight $\lambda-k\alpha_i$ with $u_k\in \ker F_i$, and $E_i^{(N)}=E_i^{(N)}/{[N]_q!}$ are the divided powers. Then the Kashiwara operators $\Tilde{e}_i$ and $\Tilde{f}_i$ are given by
\begin{align*}
\Tilde{e}_iu&=E_iu_0+E_i^{(2)}u_1+\cdots E_i^{(N+1)}u_N,\\
\Tilde{f}_iu&=u_1+E_iu_2 +\cdots + E_i^{(N-1)}u_N.
\end{align*}
For $u=E_ix$, note that $u_0=0$, so we have $\Tilde{e}_i\Tilde{f}_iE_ix=E_ix$.

Meanwhile, For any $b\in \bfB(\lambda)$, we have either $\Tilde{f}_ib=0$ or $\Tilde{e}_i\Tilde{f}_ib=b$. Thus applying this to the canonical basis decomposition of $E_ix$, we see that $\Tilde{e}_i\Tilde{f}_iE_ix=E_ix$ if and only if $c_b=0$ for any $b$ with $\Tilde{f}_ib=0$, and the result follows.
\end{proof}

\begin{proof}[Proof of Theorem \ref{thm:intersection}]
By Propositions~\ref{tp=rep} and \ref{OppositeDemazure},
\[
\bfB_{e,w}^{tp}(\lambda)\cap \bfB_{v,w_0}^{tp}(\lambda)=\bfB_w(\lambda)\cap \bfB^v(\lambda).
\]

For $v\leq v'\leq w'\leq w$, we have $\CB_{v',w',>0}\subset \overline{\CB_{v,w,>0}}$, so mapping into $\BP(L(\lambda))$, we must have $\bfB_{v',w'}^{tp}(\lambda)\subset \bfB_{v,w}^{tp}(\lambda)$. Using $e\leq v\leq w\leq w_0$, this implies
\[
\bfB_{v,w}^{tp}(\lambda)\subset\bfB_w(\lambda)\cap \bfB^v(\lambda).
\]
            
For the reverse inclusion, we proceed by induction on $w$ and $v$. The base case $v=e$ follows from $\bfB_{e,w}^{tp}(\lambda)=\bfB_w(\lambda)$ and $\bfB^e(\lambda)=\bfB(\lambda)$.

For the induction step, assume that for all $(v',w')$ such that $v'\leq v$, $w'\leq w$ and $(v',w')\neq (v,w)$, we have $\bfB_{v,w'}^{tp}(\lambda)=\bfB_{w'}(\lambda)\cap \bfB^{v'}(\lambda)$. Fix $i\in I$ such that $s_iw<w$ and let $b\in \bfB_w(\lambda)\cap \bfB^v(\lambda)$ be contained in some $i$-string $S$. Then $S\cap \bfB_w(\lambda)=S$. We have two cases:\\

Case 1: Suppose $S\cap \bfB^v(\lambda)=S$. Let $b'$ be the highest weight vector in $S$. Then $b'\in \bfB_{s_iw}(\lambda)$ and $\bfB_{s_iw}(\lambda)\cap \bfB^v(\lambda)$ is non-empty. By Proposition \ref{CrystalIntersection}, this implies that $v\leq s_iw$.

By the induction hypothesis, $b'\in \bfB_{v,s_iw}^{tp}(\lambda)$. By Proposition \ref{MultiplyBijection}, any element in $\CB_{v,w,>0}$ can be written as $y_i(a)gB^+$ for some $a\in \BR_{>0}$ and $gB^+\in \CB_{v,s_iw,>0}$. $\Psi(gB^+)$ has non-zero $b'$ coordinate, so expanding the exponent $y_i(a)$ and using positivity, we get $b\in \bfB_{v,w}^{tp}(\lambda)$.\\
            
Case 2: Otherwise, $S\cap \bfB^v(\lambda)=\{b\}$, $|S|\geq2$ and $b$ is the lowest weight vector in $S$. By Proposition \ref{DemazureString}, we have $s_iv<v$. Then using Proposition \ref{MultiplyBijection} again, any element in $\CB_{s_iv,w,>0}$ can be written as $x_i(a)gB^+$ for some $a\in \BR_{>0}$ and $gB^+\in \CB_{v,w,>0}$. If $b\not\in \bfB_{v,w}^{tp}(\lambda)$, then the coefficient of $b$ in $\Psi(gB^+)$ is zero. Expanding the exponent $x_i(a)$, Lemma \ref{LowestWtCoeffZero} shows that the coefficient of $b$ in $\Psi(x_i(a)gB^+)$ is still zero, so
\[
b\not\in \bfB_{s_iv,w}^{tp}(\lambda)=\bfB_w(\lambda)\cap \bfB^{s_iv}(\lambda)\supset \bfB_w(\lambda)\cap \bfB^v(\lambda).
\]
This is a contradiction, so we get $b\in \bfB_{v,w}^{tp}(\lambda)$.
\end{proof}

\begin{remark}
If $\lambda$ is singular, consider the minimal length representatives $W^J\subset W$ for the quotient $W/W_J$, where $W_J$ is the parabolic subgroup of $W$ fixing $\lambda$. Then from the proof, we see that Theorem~\ref{thm:intersection} still holds when $v,w\in W^J$.
\end{remark}

\section{Saturation of Bruhat interval polytopes}
\label{sec:saturation}

In this section, we recall the definition of Bruhat interval polytopes $P_{v,w}(\lambda)$. Then we show that for type $A$, the set of weights of $\bfB_{v,w}(\lambda)$ is saturated in $P_{v,w}(\lambda)$. This gives a useful application of the intersection theorem in Section~\ref{sec:richardson}.

We continue to assume that $\lambda\in X^+$ is regular.

\subsection{Bruhat interval polytopes}
\label{sec:saturation:definition}

\begin{defi}
Let $v\leq w$ in $W$. The Bruhat interval polytope $P_{v,w}(\lambda)$ is defined as the convex hull of $u\lambda$ for $v\leq u\leq w$ in $X_\BR$.
\end{defi}

The Bruhat interval polytope has an alternative definition using the moment map. The action of the torus on $\CB$ induces a moment map $\mu:\CB\to X_\BR$. Let $gB^+\in \CB$ and $Y$ be the closure of its torus orbit in $\CB$. We consider the image $\mu(Y)\subset X_\BR$. The following theorem is by Atiyah \cite{Ati82} and Guillemin--Sternberg \cite{GS82}.

\begin{thm}
$\mu(Y)$ is a convex polytope, and $\mu$ induces a dimension preserving bijection between torus orbits in $Y$ and the faces of $\mu(Y)$.
\end{thm}

In particular, the vertices of $\mu(Y)$ correspond to the $T$-fixed points in $Y$. By \cite{GS87}*{Proposition 5.1}, for any $gB^+\in \CB_{v,w,>0}$, the vertices of the polytope $\mu(Y)$ are the extremal weights $u\lambda$ such that $\eta_{u\lambda}\in \bfB^{tp}_{v,w}(\lambda)$, so $\mu(Y)$ is exactly the Bruhat interval polytope $P_{v,w}(\lambda)$.

Via Coxeter matroids, Bruhat interval polytopes are pseudo-Weyl polytopes and can be described as an intersection of half-spaces. Any facet of $P_{v,w}(\lambda)$ lies on a hyperplane of the form
\[
\langle \mu,u\omega_i^\vee\rangle= \langle \lambda_u,u\omega_i^\vee\rangle
\]
for some $u\in W$ and $i\in I$. See \citelist{\cite{GS87}\cite{Kam05}} for more detailed discussions on these polytopes.

\subsection{Saturation in type $A$}
\label{sec:saturation:typeA}

The main goal of Section \ref{sec:saturation} is to prove the following result.

\begin{thm}
\label{TypeASaturation}
Suppose $G$ is of type $A$ and let $\lambda\in X^+$. Then $\mu\in wt(\bfB_{v,w}(\lambda))$ if and only if $\lambda-\mu$ is in the root lattice and $\mu\in P_{v,w}(\lambda)$.
\end{thm}
        
Many of the tools come from \cite{BJK} with suitable modification.

\begin{lem}
\label{PolytopeMultipleIsWeight}
$\mu\in P_{v,w}(\lambda)$ if and only if there is an integer $N\geq 1$ such that $N\mu\in wt(\bfB_{v,w}^{tp}(N\lambda))$.
\end{lem}

\begin{proof}
By \cite{BJK}*{Lemma 5.6}, we have $\mu\in P_{e,w}(\lambda)$ if and only if there is an integer $N\geq 1$ such that $N\mu\in wt(\bfB_w(N\lambda))$. Similarly, using the opposite Demazure, $\mu\in P_{v,w_0}(\lambda)$ if and only if there is an integer $N\geq 1$ such that $N\mu\in wt(\bfB^v(N\lambda))$.

If $N\mu\in wt(\bfB_{v,w}^{tp}(N\lambda))$, by Theorem~\ref{thm:intersection}, we have
\[
N\mu\in wt(\bfB_w(N\lambda))\cap wt(\bfB^v(N\lambda)),
\]
so$N\mu\in P_{e,w}\cap P_{v,w_0}=P_{e,w}$.

For the other direction, $\mu$ is a rational convex combination $u\lambda$ for $v\leq u\leq w$. Therefore, for some $N\geq 1$, $N\mu$ is a sum of $u_i\lambda$ for $v\leq u_i\leq w$ for $i=1,\cdots N$. By the proof of \cite{BJK}*{Lemma 5.6}, the tensor product $u_1\eta_\lambda\ox \cdots \ox u_N\eta_\lambda$ in $L(\lambda)^{\ox N}$ has weight $N\mu$, and has non-zero projection to $L_w(N\lambda)\subset L(N\lambda)$. By the same argument for the opposite Demazure, the projection also belongs to $L^v(N\lambda)$. We conclude that there is a vector $\bfB_{v,w}^{tp}(N\lambda)=\bfB_w(N\lambda)\cap \bfB^v(N\lambda)$ with weight $N\mu$.
\end{proof}

Using this lemma, we get the following result which is crucial to the induction step.

\begin{lem}
\label{LineIntersectSmaller}
Let $\mu\in P_{v,w}(\lambda)$ with $v\leq s_iw<w$. Then the line $\mu+\BQ \alpha_i$ has non-empty intersection with $P_{v,s_iw}(\lambda)$. Furthermore, if $\lambda-\mu$ is in the root lattice, then we can find $\mu'\in P_{v,s_iw}(\lambda)$ such that $\lambda-\mu'$ is in the root lattice.
\end{lem}

\begin{proof}
Let $\mu\in P_{v,w}(\lambda)$. By Lemma \ref{PolytopeMultipleIsWeight}, there is a positive integer $N$ large enough such that $N\mu\in wt(\bfB_{v,w}^{tp}(N\lambda))$, so we can find $b\in \bfB_{v,w}^{tp}(N\lambda)$ with $wt(b)=N\mu$. 
            
By assumption, any element of $\CB_{v,w,>0}$ is of the form $y_i(a)gB^+$ for $gB^+\in \CB_{v,s_iw,>0}$ and $a>0$. Therefore, by the $\Psi(gB^+)$ definition of $\bfB_{v,w}^{tp}(N\lambda)$, we can find $b'\in \bfB_{v,s_iw}^{tp}(N\lambda)$ and a non-negative integer $n$ such that $b$ has non-zero coefficient in $f_i^nb'$. Then $wt(b')=wt(b)+n\alpha_i=N\mu+n\alpha_i\in P_{v,s_iw}(N\lambda)$, so we get $\mu+n\alpha_i/N\in P_{v,s_iw}(\lambda)$.
            
Now assume $\lambda-\mu$ is in the root lattice. The intersection $(\mu+\BQ \alpha_i)\cap P_{v,s_iw}(\lambda)$ is an interval $[\mu+k\alpha_i,\mu+l\alpha_i]$ with $k\geq l$. Then $\mu+k\alpha_i$ is on a facet of $P_{v,s_iw}(\lambda)$, so for some $u\in W$ and $j\in I$, we have
\[
\langle \mu+k\alpha_i, u\omega_j^\vee\rangle=\langle \lambda_u, u\omega_j^\vee\rangle.
\]
Furthermore, by maximality of $k$, we can find a facet such that $\langle \alpha_i,u\omega_j^\vee\rangle\neq0$. Then we get
\[
k\langle \alpha_i, u\omega_j^\vee\rangle=\langle \lambda_u, u\omega_j^\vee\rangle-\langle \mu, u\omega_j^\vee\rangle.
\]
Note that
\[
\lambda_u-\mu=(\lambda_u-\lambda)+(\lambda-\mu)
\]
is in the root lattice, so the right hand side is an integer. Meanwhile, in type $A$, any root can be written as a sum of simple roots with coefficients in $\{-1,0,1\}$, so $\langle \alpha_i, u\omega_j^\vee\rangle=\langle u^{-1}\alpha_i,\omega_j^\vee\rangle$ is $\pm1$. This implies $k$ is an integer. Thus we get $\mu'=\mu+k\alpha_i\in P_{v,s_iw}(\lambda)$ such that $\lambda-\mu'$ is in the root lattice.
\end{proof}

We are now ready to prove Theorem~\ref{TypeASaturation}.

\begin{proof}[Proof of Theorem~\ref{TypeASaturation}]
We proceed by induction similar to the proof of Theorem \ref{thm:intersection}. The base case $v=w=e$ is trivial. Assume that the theorem is true for all $(v',w')$ with $v' \leq v$, $w'\leq w$ and $(v',w')\neq (v,w)$. Fix $i\in I$ such that $s_iw<w$. We have three cases:\\

Case 1: If $s_iv<v$ and $v\not\leq s_iw$, then $P_{v,w}(\lambda)=s_iP_{s_iv,s_iw}(\lambda)$ and any element of $\CB_{v,w,>0}$ is of the form $\dot{s}_igB^+$ for some $gB^+\in \CB_{s_iv,s_iw,>0}$. Since $\dot{s}_i$ acts on $L(\lambda)$ by sending weight vectors with weight $\mu$ to weight vectors with weight $s_i\mu$, the theorem follows from induction hypothesis.\\

Case 2: If $s_iv>v$, then $P_{v,w}(\lambda)$ is stable under $s_i$. By Lemma \ref{LineIntersectSmaller}, we can find $\mu'\in P_{v,s_iw}(\lambda)$ with $\lambda-\mu'$ in the root lattice, and $\mu'$ is maximal. Then $\mu'\geq \mu$ and since $s_i\mu\in P_{v,w}(\lambda)$, we also have $\mu'\geq s_i\mu$. This forces $\mu'\geq \mu\geq s_i\mu'$.

By induction hypothesis, we can find $b'\in \bfB_{v,s_iw}^{tp}(\lambda)$ with $wt(b')=\mu'$. Then the $i$-string containing $b'$ lies in $\bfB_{v,w}^{tp}(\lambda)$, and we can find $b$ in the $i$-string between $b'$ and $\Tilde{s}_ib'$ with $wt(b)=\mu$.
            
Conversely, for any $\mu=wt(b)\in wt(\bfB_{v,w}^{tp}(\lambda))$, the highest weight in the $i$-string of $b$ lies in $P_{v,s_iw}(\lambda)$ by assumption, and the lowest weight lies in $s_i(P_{v,s_iw}(\lambda))\subset P_{v,w}(\lambda)$. By convexity $\mu\in P_{v,w}(\lambda)$.\\
            
Case 3: If $s_iv<v$ and $v\leq s_iw$, then $\mu\in P_{s_iv,w}(\lambda)$. By the induction hypothesis, we can find $b\in \bfB_{s_iv,w}^{tp}(\lambda)$ such that $wt(b)=\mu$. Since any element in $\CB_{s_iv,w,>0}$ can be written as $x_i(a)gB^+$ for $a>0$ and $gB^+\in \CB_{v,w,>0}$, so by $b\in \bfB_{s_iv,w}^{tp}(\lambda)$, there is $b'\in \bfB_{v,w}^{tp}(\lambda)$ and a non-negative integer $n$ such that $b$ has non-zero coefficient in $e_i^nb$.
            
By the induction hypothesis, $(\mu+\BQ\alpha_i)\cap wt(\bfB_{v,s_iw}(\lambda))$ is a consecutive set of weights. Meanwhile, $\bfB_{v,w}^{tp}(\lambda)$ is the union of the canonical basis vectors with non-zero coefficients in $f_i^nb$ as $n$ runs over the non-negative integers and $b$ runs over $\bfB_{v,s_iw}^{tp}(\lambda)$. Therefore, $(\mu+\BQ\alpha_i)\cap wt(\bfB_{v,w}^{tp}(\lambda))$ is a consecutive set of weights. Since $\mu$ is in between two weights: the maximal $\mu'$ from Lemma \ref{LineIntersectSmaller} and $wt(b')$, this shows that $\mu\in wt(\bfB_{v,w}^{tp}(\lambda))$.

Conversely, let $\mu=wt(b)\in wt(\bfB_{v,w}^{tp}(\lambda))$. The maximum $\mu'$ in the intersection $(\mu+\BQ\alpha_i)\cap P_{v,w}(\lambda)$ belongs to $P_{v,s_iw}(\lambda)$ must satisfy $\mu'\geq \mu$ by the generation of $\bfB_{v,w}^{tp}(\lambda)$ from $\bfB_{v,s_iw}^{tp}(\lambda)$.

By the induction hypothesis for $(s_iv,w)$, $\mu\in P_{s_iv,w}(\lambda)$. For any $s_iv\leq u\leq w$, if $s_iu>u$, then $v\leq s_iu\leq w$, while if $s_iu<u$, then $v\leq u\leq w$. In any case, $u\lambda=u'\lambda+k\alpha_i$ with $k\geq0$ and $v\leq u'\leq w$. Writing $\mu$ as a convex combination of the vertices in $P_{s_iv,w}(\lambda)$, we see that $\mu=\mu''+k'\alpha_i$ for some $\mu''\in P_{v,w}(\lambda)$, and $k'\in \BQ_{\geq0}$. We get $\mu''\leq \mu\leq \mu'$ with $\mu',\mu''\in P_{v,w}(\lambda)$, so $\mu\in P_{v,w}(\lambda)$.
\end{proof}

\end{document}